\documentclass[12pt]{amsart}

\usepackage[utf8]{inputenc}
\usepackage[T1]{fontenc}
\usepackage{amsmath, amsthm, amssymb}
\usepackage[left=1.2in,top=1in,right=1in,bottom=1in]{geometry}
\usepackage[colorlinks=true,linkcolor=blue,citecolor=blue,urlcolor=blue]{hyperref}

\numberwithin{equation}{section}

\newtheorem{theorem}{Theorem}[section]
\newtheorem{proposition}[theorem]{Proposition}
\newtheorem{lemma}[theorem]{Lemma}
\newtheorem{corollary}[theorem]{Corollary}
\theoremstyle{definition}
\newtheorem{definition}[theorem]{Definition}
\newtheorem{ex}[theorem]{Example}
\theoremstyle{remark}
\newtheorem{rem}[theorem]{Remark}

\DeclareMathOperator{\Scal}{Scal}
\DeclareMathOperator{\Ker}{Ker}
\DeclareMathOperator{\im}{Im}
\DeclareMathOperator{\coker}{Coker}
\DeclareMathOperator{\Ad}{Ad}
\newcommand{\spinS}{\mathcal{S}}
\newcommand{\dirac}{\mathcal{D}}
\newcommand{\conn}{\nabla}
\newcommand{\gfrak}{\mathfrak{g}}
\newcommand{\cliff}{\cdot}
\newcommand{\hook}{\lrcorner}

\title{Deformations of Locally Conformal Parallel $Spin(7)$ Instantons}

\author{Ey\"up Yal\c{c}\i nkaya}
\address{Bilkent University, Department of Mathematics, Ankara, T\"urkiye}
\email{eyup.yalcinkaya@bilkent.edu.tr}
\address{The Scientific and Technological Research Council of T\"urkiye (T\"UB\.ITAK), Ankara, T\"urkiye}
\email{eyup.yalcinkaya@tubitak.gov.tr}

\thanks{The author thanks Sema Salur for helpful advice and T\"UB\.ITAK for its support.}


\begin{document}

\maketitle

\begin{abstract}
We study the infinitesimal deformation theory of $Spin(7)$-instantons on compact locally conformal parallel (LCP) $Spin(7)$ manifolds, i.e.\ $Spin(7)$ manifolds whose fundamental 4-form satisfies $d\Phi=\theta\wedge\Phi$ with a closed, non-exact Lee form $\theta$. We show that the gauge-fixed linearized instanton equations are equivalent to a kernel problem for a twisted Levi-Civita Dirac operator $\dirac^-_A-\mathcal{A}_\theta:\Gamma(\spinS^-\otimes\gfrak_P)\to\Gamma(\spinS^+\otimes\gfrak_P)$, compute the zeroth-order term $\mathcal{A}_\theta$ explicitly, and show that its pointwise squared operator norm equals $\frac{49}{64}|\theta|^2$. Combined with the Lichnerowicz--Weitzenb\"ock formula and Ivanov's scalar curvature identity, this yields a sufficient condition for rigidity, which we show can never be satisfied by flat connections. Indeed, the trivial connection is never rigid: its deformation space contains $H^1(M;\mathbb{R})\otimes\gfrak$. Our main result concerns the product LCP manifolds $N\times S^1$, where $N$ is a compact nearly parallel $G_2$ manifold: for the pull-back of a $G_2$-instanton $B$ on $N$ we prove $\mathcal{H}^1_{Spin(7)}(\mathrm{pr}^*B)\cong\mathcal{H}^0_{G_2}(B)\oplus\mathcal{H}^1_{G_2}(B)$. This reduces the deformation theory of such instantons to the nearly parallel $G_2$ case, shows that the inclusion above is an equality for the trivial connection, and produces non-flat examples on the Hopf manifold $S^7\times S^1$ and on other products with homogeneous nearly parallel $G_2$ manifolds.
\end{abstract}


\section{Introduction}

Manifolds with special holonomy, particularly those with exceptional holonomy $G_2$ (in 7 dimensions) and $Spin(7)$ (in 8 dimensions), are of fundamental importance in modern mathematics and theoretical physics. They are central to M-theory and string theory compactifications, where they provide a pathway to supersymmetric theories in lower dimensions \cite{tanaka_construction_2012}. While torsion-free (integrable) $G$-structures, where the holonomy $Hol(g)$ is a subgroup of $G$, are the most well-studied \cite{joyce_riemannian_nodate, bryant_construction_1989}, both physical models and pure geometric inquiry often lead to the study of non-integrable $G$-structures, where torsion is present.

Among the most structured and tractable non-integrable cases are the locally conformal parallel (LCP) $Spin(7)$ manifolds. These structures, belonging to the $W_2$ class in the Fern\'andez classification \cite{fernandez_classification_1986}, offer a rich middle ground. They are not integrable, as their fundamental 4-form $\Phi$ is not parallel ($\nabla^{LC} \Phi \neq 0$). However, their non-integrability is precisely controlled by a closed 1-form $\theta$, the \textbf{Lee form}, via the relation $d\Phi = \theta \wedge \Phi$. Compact non-parallel LCP $Spin(7)$ manifolds are well understood structurally: by Ivanov--Parton--Piccinni \cite{ivanov_parton_piccinni_2006} they fibre over $S^1$, and the basic example is the Hopf manifold $S^7\times S^1$. This makes them ideal test cases for extending gauge-theoretic constructions from the parallel (torsion-free) setting to the setting with torsion.

A natural gauge-theoretic object to study on these manifolds is the $Spin(7)$-instanton, a generalization of the anti-self-dual (ASD) Yang--Mills equations from 4-dimensional geometry \cite{atiyah_deformations_1977, donaldson_thomas_1998, reyes_carrion_1998}. The spinorial approach to deformations of instantons on manifolds with real Killing spinors was developed by Charbonneau and Harland \cite{charbonneau_deformations_2016} in the nearly K\"ahler case and by Singhal \cite{singhal_deformations_2022} in the nearly parallel $G_2$ case; deformations of $Spin(7)$-instantons on Bryant--Salamon manifolds were recently studied by Ghosh \cite{ghosh_deformations_2025}. This paper is dedicated to analyzing the infinitesimal deformation theory of $Spin(7)$-instantons in the presence of locally conformal torsion.

Our first main result (Proposition \ref{prop:lcspin7_main} and Corollary \ref{cor:H1}) reformulates the linearized deformation problem. On a compact LCP $Spin(7)$ manifold the infinitesimal deformation space $\mathcal{H}^1$ is the first cohomology of a three-term elliptic complex, and by Hodge theory it is identified with the kernel of a twisted Dirac operator $\dirac^-_{A} - \mathcal{A}_\theta$ acting between spinor bundles of opposite chirality ($\spinS^- \to \spinS^+$). In Proposition \ref{prop:Atheta} we compute the zeroth-order term $\mathcal{A}_\theta$ explicitly in terms of the Lee form and show that its pointwise squared operator norm is $\frac{49}{64}|\theta|^2$.

We then use the Lichnerowicz--Weitzenb\"ock formula to obtain a sufficient condition for rigidity (Proposition \ref{prop:rigidity}): $A$ is rigid if the lowest eigenvalue $\lambda_A$ of its Clifford curvature endomorphism satisfies
$$\lambda_A > \frac{7}{64}|\theta|^2 - \frac{7}{8}\,\delta\theta$$
pointwise on $M$. Since the right-hand side has positive integral, this criterion never applies to flat connections. This is consistent with Theorem \ref{thm:nonrigid}: on every compact non-parallel LCP $Spin(7)$ manifold the trivial flat connection is non-rigid, with $\mathcal{H}^1 \supseteq H^1(M;\mathbb{R})\otimes\gfrak \neq 0$, and for $G = U(1)$ these infinitesimal deformations integrate to a torus of flat instantons.

By Ivanov--Parton--Piccinni \cite{ivanov_parton_piccinni_2006}, the simplest compact LCP $Spin(7)$ manifolds are the products $M=N\times S^1$ with $N$ a compact nearly parallel $G_2$ manifold. Our main result (Theorem \ref{thm:product}) shows that on such products the deformation theory of pulled-back instantons reduces completely to the seven-dimensional theory studied by Singhal \cite{singhal_deformations_2022}: if $B$ is a $G_2$-instanton on $N$, then
$$\mathcal{H}^1_{Spin(7)}(M,\mathrm{pr}^*B)\cong\mathcal{H}^0_{G_2}(N,B)\oplus\mathcal{H}^1_{G_2}(N,B).$$
The proof uses the observation that, in temporal form, the gauge-fixed linearized equations on the cylinder become an evolution equation $\partial_t u = L_Bu$ for a formally self-adjoint elliptic operator on $N$, which forces all non-zero Fourier modes along $S^1$ to vanish. As consequences, $\mathcal{H}^1(A_0)=H^1(M;\mathbb{R})\otimes\gfrak$ exactly for the trivial connection on $N\times S^1$ (Corollary \ref{cor:trivial_product}), and every pull-back instanton is non-rigid as soon as $B$ is reducible. Applied to the canonical connections on homogeneous nearly parallel $G_2$ manifolds, this gives non-flat LCP $Spin(7)$-instantons whose deformation spaces are computable from the $G_2$ data.

All constants are illustrated on the Hopf manifold $S^7\times S^1$, where the Weitzenb\"ock identity can be checked by hand.

\section{Preliminaries on (LC) $Spin(7)$-Structures}

\subsection{Fundamental Definitions and Decompositions}

Let $(x^1,..., x^8)$ be coordinates on $\mathbb{R}^8$.
The standard Cayley 4-form on $\mathbb{R}^8$ is given by
\begin{align*}
\Phi_0&=dx^{1234}+dx^{1256}+dx^{1278}+dx^{1357}-dx^{1368}-dx^{1458}-dx^{1467}\\
&-dx^{2358}-dx^{2367}-dx^{2457}+dx^{2468}+dx^{3456}+dx^{3478}+dx^{5678}
\end{align*}
where $dx^{ijkl}=dx^i\wedge dx^j\wedge dx^k \wedge dx^l$.
The subgroup of $GL(8, \mathbb{R})$ that preserves $\Phi_0$ (and the induced metric $g_0$) is the Lie group $Spin(7)$.

\begin{definition}
A $Spin(7)$-structure on an 8-dimensional manifold $M$ is a reduction of the structure group of its frame bundle from $GL(8, \mathbb{R})$ to $Spin(7)$.
This is equivalent to the global existence of an admissible 4-form $\Phi \in \Omega^4(M)$, meaning at each $p \in M$, the pair $(\Phi_p, g_p)$ is algebraically isomorphic to $(\Phi_0, g_0)$ on $\mathbb{R}^8$.
\end{definition}

The existence of a $Spin(7)$-structure is a topological condition \cite{lawson_spin_2016} requiring the first and second Stiefel-Whitney classes to vanish ($w_1(M)=w_2(M)=0$) and a condition on the Pontryagin classes and Euler characteristic ($p_1(M)^2-4p_2(M)\pm 8\chi(M)=0$).
If the structure is integrable, i.e., $\Phi$ is parallel with respect to the Levi-Civita connection ($\nabla^{LC}\Phi=0$), the holonomy is contained in $Spin(7)$ and $M$ is a Ricci-flat 8-manifold \cite{joyce_riemannian_nodate}.

Since $Spin(7)\subset SO(8)$ lifts to $Spin(8)$, a $Spin(7)$-structure determines a spin structure, and $Spin(7)$ is precisely the stabilizer of a unit positive spinor. We denote by $\phi \in \Gamma(\spinS^+)$ the unit spinor (unique up to sign) defining the structure. Throughout, Clifford multiplication satisfies $X\cliff Y + Y\cliff X = -2g(X,Y)$, and forms act on spinors through the standard identification $Cl(TM)\cong\Lambda^\bullet T^*M$, so that $X\cliff \psi$ corresponds to $(X\wedge - X\hook)$.

The $Spin(7)$-structure induces an orthogonal decomposition of the $k$-form bundles into irreducible $Spin(7)$-representations.
For instanton theory, the most relevant decompositions are for 2-forms and 3-forms:
\begin{align}
\Lambda^2(M) &= \Lambda^2_7\oplus \Lambda^2_{21}, \\
\Lambda^3(M) &= \Lambda^3_8\oplus\Lambda^3_{48}. \label{dec}
\end{align}
Here, $\Lambda^2_7$ and $\Lambda^2_{21}$ are the 7- and 21-dimensional irreducible representations; $\Lambda^2_{21}\cong\mathfrak{spin}(7)$ and $\Lambda^2_{21} = \{\alpha \in \Lambda^2(M)\,|\,*(\alpha\wedge\Phi)=-\alpha\}$.
The instanton condition (see Section 3.1) is precisely the constraint that the curvature $F_A$ has no component in the $\Lambda^2_7$ sub-bundle.

We shall repeatedly use the following standard algebraic facts (see e.g.\ \cite{lawson_spin_2016}).

\begin{lemma}
\label{lem:spinor_decomp}
Let $\phi$ be the unit spinor defining the $Spin(7)$-structure. Then:
\begin{enumerate}
\item[(a)] $\alpha\cliff\phi = 0$ for every $\alpha\in\Lambda^2_{21}$ (this is the infinitesimal form of the statement that $Spin(7)$ fixes $\phi$);
\item[(b)] the map $\Lambda^2_7 \to \phi^\perp\subset \spinS^+$, $\omega\mapsto\omega\cliff\phi$, is a linear isomorphism, so that
$$\spinS^+ = \mathbb{R}\phi \oplus \Lambda^2_7\cliff\phi,\qquad \dim = 1+7;$$
\item[(c)] the map $\Lambda^1 \to \spinS^-$, $a\mapsto a\cliff\phi$, is an isometric isomorphism.
\end{enumerate}
\end{lemma}

\begin{proof}
(a) holds since $\Lambda^2_{21}\cong\mathfrak{spin}(7)$ acts on spinors through the Clifford action of 2-forms and annihilates its invariant vector $\phi$. For (b), $\omega\cliff\phi$ is orthogonal to $\phi$ because 2-forms act skew-adjointly; the map is $Spin(7)$-equivariant and non-zero, and both spaces are irreducible 7-dimensional $Spin(7)$-modules, so by Schur's lemma it is an isomorphism. For (c), $|a\cliff\phi|^2 = -\langle a\cliff a\cliff \phi,\phi\rangle = |a|^2$, and $\dim\Lambda^1 = \dim\spinS^- = 8$.
\end{proof}

\subsection{Locally Conformal Structures and Torsion}
We focus on non-integrable structures where $\nabla^{LC}\Phi \neq 0$.
According to the Fern\'andez classification \cite{fernandez_classification_1986}, the $W_2$ class, also known as locally conformal parallel (LCP) structures, is defined by
$$d\Phi = \theta \wedge \Phi,\qquad d\theta = 0,$$
where the closed 1-form $\theta$ is the Lee form. Locally $\theta = 4\,df$ for a function $f$, and $e^{-4f}\Phi$ is a parallel $Spin(7)$-structure for the conformal metric $e^{-2f}g$. We call the structure \emph{non-parallel} (or strict) if $\theta$ is not exact. On a compact manifold this implies $b_1(M)\ge1$; in fact Ivanov, Parton and Piccinni \cite[Theorem B]{ivanov_parton_piccinni_2006} showed that every compact non-parallel LCP $Spin(7)$ manifold fibres over $S^1$. In particular, compact non-parallel LCP $Spin(7)$ manifolds are never simply connected.

A key tool for working with $Spin(7)$ manifolds was provided by Ivanov \cite{ivanov_connection_2004}, who proved the existence of a \emph{unique} linear connection $\nabla$, the characteristic connection, which preserves the $Spin(7)$-structure ($\nabla \Phi = \nabla g = 0$) and has totally skew-symmetric torsion $T \in \Lambda^3(M)$.
For LCP $Spin(7)$ manifolds the torsion satisfies $T \in \Lambda^3_8$ (i.e., $T_{48}=0$). Substituting $d\Phi = \theta \wedge \Phi$ (which implies $\delta\Phi = -\ast(\theta \wedge \Phi)$) into Ivanov's formula (\ref{eq:ivanovT}) below yields
\begin{equation}
\label{eq:TLCP}
T = -\frac{1}{6}\ast (\theta \wedge \Phi).
\end{equation}

\noindent\textbf{Norm conventions.} For a $k$-form $\beta$ we write $|\beta|^2$ for the norm in which $|e^{i_1}\wedge\dots\wedge e^{i_k}|=1$, and $\|\beta\|^2 = \sum_{i_1,\dots,i_k}\beta_{i_1\dots i_k}^2 = k!\,|\beta|^2$ for the tensor norm used by Ivanov. Since $|\theta\wedge\Phi|^2 = |\theta|^2|\Phi|^2 - |\theta\hook\Phi|^2 = 14|\theta|^2 - 7|\theta|^2$, equation (\ref{eq:TLCP}) gives
\begin{equation}
\label{eq:Tnorm}
\|T\|^2 = 6|T|^2 = \frac{7}{6}|\theta|^2 .
\end{equation}

The following theorem by Ivanov relates the Levi-Civita geometry to the characteristic geometry.

\begin{theorem}[Ivanov \cite{ivanov_connection_2004}]
\label{thm:ivanov}
Let $(M, g, \Phi)$ be an 8-dimensional $\mathrm{Spin}(7)$ manifold.
\begin{enumerate}
\item[(i)] There exists a unique characteristic connection $\nabla$ with skew-symmetric torsion $T$ given by
\begin{equation}
\label{eq:ivanovT}
T = -\delta\Phi - \frac{7}{6} * (\theta \wedge \Phi), \quad \theta = \frac{1}{7} * (\delta\Phi \wedge \Phi).
\end{equation}
The spinor $\phi$ is parallel for the characteristic connection, $\nabla\phi=0$, and
\begin{equation}
T \cdot \phi = -\frac{7}{6} \theta \cdot \phi.
\end{equation}

\item[(ii)] The Levi-Civita scalar curvature $\mathrm{Scal}_g$ satisfies
\begin{equation}
\label{eq:ivanovScal}
\mathrm{Scal}_g = \frac{49}{18} \|\theta\|^2 - \frac{1}{12} \|T\|^2 + \frac{7}{2} \delta\theta .
\end{equation}
\end{enumerate}
\end{theorem}

In the LCP case, inserting (\ref{eq:Tnorm}) into (\ref{eq:ivanovScal}) (note $\|\theta\|=|\theta|$ for 1-forms) gives
\begin{equation}
\label{eq:ScalLCP}
\Scal_g = \frac{21}{8}|\theta|^2 + \frac{7}{2}\,\delta\theta .
\end{equation}
As a check, (\ref{eq:ScalLCP}) agrees with the classical conformal-change formula $\Scal_{e^{2u}g_0} = 14\,\delta_g du + 42|du|_g^2$ for a Ricci-flat $g_0$ in dimension $8$, with $g = e^{2u}g_0$, $u=f$ and $\theta = 4\,du$.

\begin{ex}[The Hopf manifold]
\label{ex:hopf}
Let $\Phi_0$ be the Cayley form on $\mathbb{R}^8\setminus\{0\}$ and $r=|x|$. The metric $g = r^{-2}g_0 = dt^2 + g_{S^7}$, $t=\log r$, is the product of the round unit sphere with a line, and $\Phi = r^{-4}\Phi_0$ is a $Spin(7)$-form for $g$. Both are invariant under $x\mapsto\lambda x$ ($\lambda>1$), hence descend to $M = (\mathbb{R}^8\setminus\{0\})/\lambda^{\mathbb{Z}} \cong S^7\times S^1$. Since $d\Phi = -4r^{-5}dr\wedge\Phi_0 = -4\,dt\wedge\Phi$, this is a compact non-parallel LCP $Spin(7)$ manifold with
$$\theta = -4\,dt,\qquad |\theta|^2 = 16,\qquad \delta\theta = 0,\qquad \Scal_g = 42,$$
in agreement with (\ref{eq:ScalLCP}) and with the scalar curvature $42$ of the round $S^7$. Here $b_1(M)=1$ and $H^1(M;\mathbb{R})$ is spanned by $[dt]$.
\end{ex}

\section{The $Spin(7)$-Instanton Deformation Problem}

\subsection{Definition of $Spin(7)$-Instantons}
Let $G$ be a compact Lie group, $P \to M$ a principal $G$-bundle over an LCP $Spin(7)$-manifold $M$, and $\gfrak_P = P\times_{\Ad}\gfrak$ its adjoint bundle, equipped with an $\Ad$-invariant inner product.
Let $A$ be a connection on $P$ with curvature $F_A \in \Omega^2(M, \gfrak_P)$.

\begin{definition}[$Spin(7)$-Instanton]
A connection $A$ is a $Spin(7)$-instanton if its curvature $F_A$ lies entirely in the $\Lambda^2_{21}$ sub-bundle, that is,
$$\pi_7(F_A) = 0,$$
where $\pi_7: \Lambda^2 \to \Lambda^2_7$ is the orthogonal projection.
\end{definition}

This is a direct generalization of the ASD Yang--Mills equations in 4 dimensions, where $\Lambda^2 = \Lambda^2_+ \oplus \Lambda^2_-$ and the instanton condition is $\pi_+(F_A) = 0$. Note that the definition only involves the $Spin(7)$-structure pointwise; in particular every flat connection is a $Spin(7)$-instanton, regardless of the torsion.

By Lemma \ref{lem:spinor_decomp}, the instanton condition is equivalent to $F_A \cdot \phi = 0$, where $\phi$ is the defining spinor \cite{lawson_spin_2016}.
This spinorial approach provides a powerful framework for defining instantons on various $G$-structures \cite{harland_instantons_2012}, though care must be taken in non-integrable settings. For instance, on manifolds admitting Killing spinors, Harland and N\"olle \cite{harland_instantons_2012} showed that the spinorial instanton condition is sufficient to imply the full Yang--Mills equations. Although their precise analytical results do not directly carry over to the locally conformal $Spin(7)$ setting due to differing torsion geometries, their framework serves as an important conceptual motivation.

\subsection{The Deformation Complex}
Let $\mathcal{G}$ denote the gauge group. An infinitesimal deformation $a \in \Omega^1(M, \gfrak_P)$ of an instanton $A$ must satisfy the linearized instanton equation
$$\pi_7(d_A a) = 0,$$
obtained by differentiating $\pi_7(F_{A+sa}) = 0$ at $s=0$. Trivial deformations are those generated by infinitesimal gauge transformations, $a = d_A u$ for $u \in \Omega^0(M, \gfrak_P)$. Since $\pi_7(d_Ad_Au) = [\pi_7(F_A),u] = 0$, we obtain a complex
\begin{equation}
\label{eq:complex}
0 \to \Omega^0(M, \gfrak_P) \xrightarrow{L_0} \Omega^1(M, \gfrak_P) \xrightarrow{L_1} \Omega^2_7(M, \gfrak_P) \to 0,
\end{equation}
with $L_0 = d_A$ and $L_1 = \pi_7 \circ d_A$. Its cohomology groups are:
\begin{enumerate}
\item $\mathcal{H}^0 = \ker(L_0)$, the Lie algebra of the stabilizer of $A$ in $\mathcal{G}$. It vanishes if $A$ is irreducible.
\item $\mathcal{H}^1 = \ker(L_1) / \im(L_0)$, the space of infinitesimal deformations.
\item $\mathcal{H}^2 = \coker(L_1)$, which contains the obstructions to integrating infinitesimal deformations.
\end{enumerate}
When $A$ is irreducible and $\mathcal{H}^2=0$, standard Kuranishi theory shows that the moduli space $\mathcal{M}$ of instantons is a smooth manifold near $[A]$ with tangent space $\mathcal{H}^1$. In general (e.g.\ at reducible connections such as flat ones, where $\mathcal{H}^0\neq0$) one only has a Kuranishi model, and a non-zero $\mathcal{H}^1$ does not by itself imply that $\mathcal{M}$ is non-trivial near $[A]$. We call $A$ \emph{infinitesimally rigid} (or simply \emph{rigid}) if $\mathcal{H}^1 = 0$.

The complex (\ref{eq:complex}) is elliptic: its symbol at each point depends only on the $Spin(7)$-structure at that point and therefore coincides with the symbol of the deformation complex of $Spin(7)$-instantons on flat $\mathbb{R}^8$, which is elliptic \cite{reyes_carrion_1998}. Alternatively, ellipticity follows from Proposition \ref{prop:lcspin7_main} below, which identifies $L_1\oplus L_0^*$ with a Dirac operator up to zeroth-order terms.

\section{The Dirac Operator Equivalence}

\noindent\textbf{Standing assumption.} From now on $M$ is a \emph{compact} LCP $Spin(7)$ manifold without boundary, $\phi$ is the unit defining spinor, and $A$ is a $Spin(7)$-instanton on $P\to M$.

\medskip

Since $M$ is compact and (\ref{eq:complex}) is elliptic, Hodge theory for elliptic complexes gives an $L^2$-orthogonal decomposition $\ker L_1 = \im L_0 \oplus (\ker L_1\cap\ker L_0^*)$, and hence
\begin{equation}
\label{eq:hodge}
\mathcal{H}^1 \cong \{a\in\Omega^1(M,\gfrak_P) : \pi_7(d_Aa) = 0,\ d_A^*a = 0\}.
\end{equation}
Indeed, $\im L_0$ is closed and every $a\in\ker L_1$ decomposes uniquely as $a = d_Au + a_0$ with $d_A^* a_0 = 0$, where $u$ solves $d_A^*d_Au = d_A^*a$; then $\pi_7(d_Aa_0) = \pi_7(d_Aa) = 0$.
Thus we must solve the elliptic system
\begin{align}
\pi_7(d_A a) &= 0 \quad (\text{linearized instanton equation}), \label{eq:lin_inst} \\
d_A^* a &= 0 \quad (\text{Coulomb gauge}). \label{eq:gauge_fix}
\end{align}

Let $\conn^A$ denote the connection on $\spinS\otimes\gfrak_P$ induced by the Levi-Civita spin connection and $A$, and let
$$\dirac_A = \sum_{i=1}^8 e_i\cliff\conn^A_{e_i},\qquad \dirac^\pm_A = \dirac_A|_{\spinS^\pm\otimes\gfrak_P}:\Gamma(\spinS^\pm\otimes\gfrak_P)\to\Gamma(\spinS^\mp\otimes\gfrak_P),$$
be the twisted Levi-Civita Dirac operator and its chiral parts. Here $\{e_i\}$ is a local orthonormal frame. Note that $\dirac_A$ is formally self-adjoint and $(\dirac_A^-)^* = \dirac^+_A$.

\begin{proposition}[Deformation equations vs.\ twisted Dirac operator]
\label{prop:lcspin7_main}
For $a \in \Omega^1(M, \gfrak_P)$ put $v = a \cliff \phi \in \Gamma(\spinS^- \otimes \gfrak_P)$, and define the zeroth-order operator
\begin{equation}
\label{eq:Adef}
\mathcal{A}_\theta(v) = \sum_{i=1}^8 e_i \cliff a \cliff \nabla^{LC}_{e_i} \phi \ \in\ \Gamma(\spinS^+ \otimes \gfrak_P),
\end{equation}
which is well defined as a function of $v$ by Lemma \ref{lem:spinor_decomp}(c). Then
\begin{equation}
\label{eq:key}
\dirac^-_A v - \mathcal{A}_\theta(v) = \pi_7(d_Aa)\cliff\phi + (d_A^*a)\,\phi ,
\end{equation}
and consequently $a$ solves (\ref{eq:lin_inst})--(\ref{eq:gauge_fix}) if and only if
$$\dirac^-_A v - \mathcal{A}_\theta(v) = 0 .$$
\end{proposition}

\begin{proof}
By the Leibniz rule,
\begin{equation}
\label{eq:leibniz}
\dirac^-_A v = \sum_{i=1}^8 e_i \cliff (\conn^A_{e_i} a) \cliff \phi + \sum_{i=1}^8 e_i \cliff a \cliff \nabla^{LC}_{e_i} \phi .
\end{equation}
Under the identification $Cl(TM)\cong\Lambda^\bullet T^*M$ we have $\sum_i e_i\cliff \conn^A_{e_i}a = \sum_i e_i\wedge\conn^A_{e_i}a - \sum_i e_i\hook\conn^A_{e_i}a = d_Aa + d_A^*a$, since $\nabla^{LC}$ is torsion-free. Hence the first term equals $(d_Aa)\cliff\phi + (d_A^*a)\phi$. Writing $d_A a = \pi_7(d_A a) + \pi_{21}(d_A a)$ and using Lemma \ref{lem:spinor_decomp}(a), $\pi_{21}(d_Aa)\cliff\phi = 0$, which proves (\ref{eq:key}).

By Lemma \ref{lem:spinor_decomp}(b), $\pi_7(d_Aa)\cliff\phi$ is a section of $\phi^\perp\otimes\gfrak_P$, while $(d_A^*a)\phi$ is a section of $\mathbb{R}\phi\otimes\gfrak_P$. These are pointwise orthogonal, so the right-hand side of (\ref{eq:key}) vanishes if and only if both terms vanish. Moreover $(d_A^*a)\phi = 0$ iff $d_A^*a=0$ since $|\phi|=1$, and $\pi_7(d_Aa)\cliff\phi = 0$ iff $\pi_7(d_Aa)=0$ by the injectivity in Lemma \ref{lem:spinor_decomp}(b).

Finally, the chiralities are consistent: $\dirac^-_A$ maps $\spinS^-$ to $\spinS^+$, and since $\nabla^{LC}_{e_i}\phi\in\spinS^+$, multiplication by the two vectors $e_i$ and $a$ keeps $\mathcal{A}_\theta(v)$ in $\spinS^+$.
\end{proof}

\begin{corollary}
\label{cor:H1}
Under the standing assumption, the map $a\mapsto a\cliff\phi$ induces an isomorphism
$$\mathcal{H}^1 \cong \Ker\left( \dirac^-_{A} - \mathcal{A}_\theta \right)\subset\Gamma(\spinS^-\otimes\gfrak_P).$$
In particular $\mathcal{H}^1$ is finite dimensional, and the operator $L_1\oplus L_0^*$ is elliptic.
\end{corollary}

\begin{proof}
Combine (\ref{eq:hodge}), Proposition \ref{prop:lcspin7_main} and Lemma \ref{lem:spinor_decomp}(b),(c): the isomorphisms $\Lambda^1\cong\spinS^-$ and $\Lambda^0\oplus\Lambda^2_7\cong\spinS^+$ intertwine $L_1\oplus L_0^*$ with $\dirac^-_A - \mathcal{A}_\theta$, which is a Dirac operator plus a zeroth-order term, hence elliptic.
\end{proof}

\subsection{The torsion term}
We now compute $\mathcal{A}_\theta$. For the characteristic spin connection with skew torsion $T$ one has, with the usual convention, $\nabla_X\phi = \nabla^{LC}_X\phi + \frac14 (X\hook T)\cliff\phi$, so $\nabla\phi=0$ gives
\begin{equation}
\label{eq:nablaLCphi}
\nabla^{LC}_X \phi = -\frac{1}{4} (X\hook T) \cliff \phi .
\end{equation}
In the LCP case it is more transparent to compute $\nabla^{LC}\phi$ directly from the local conformal flatness of the structure.

\begin{lemma}
\label{lem:nablaphi}
Let $\xi = \frac14\theta^\sharp$. Then for every vector field $X$,
\begin{equation}
\label{eq:nablaphi_xi}
\nabla^{LC}_X \phi = -\frac12\left(X\cliff\xi + g(X,\xi)\right)\cliff\phi ,
\qquad
|\nabla^{LC}\phi|^2 = \frac{7}{4}|\xi|^2 = \frac{7}{64}|\theta|^2 .
\end{equation}
\end{lemma}

\begin{proof}
Locally $\theta = 4df$ and $g = e^{2u}g_0$ with $u = f$, where $g_0$ carries the parallel $Spin(7)$-structure $\Phi_0 = e^{-4f}\Phi$ with parallel unit spinor $\phi_0$. Under the standard identification of spinor bundles of conformal metrics \cite{hitchin_harmonic_1974}, $\phi_0$ corresponds to the unit spinor $\phi$ defining $\Phi$, and
$$\nabla^{g}_X\phi = \widetilde{\nabla^{g_0}_X\phi_0} - \frac12 X\cliff \mathrm{grad}_g u\cliff\phi - \frac12 X(u)\,\phi .$$
Since $\nabla^{g_0}\phi_0=0$ and $\mathrm{grad}_g u = \frac14\theta^\sharp = \xi$, the formula follows. For the norm, in an orthonormal frame $e_i\cliff\xi + \xi_i = \sum_{j\neq i}\xi_j\, e_i\cliff e_j$, and the spinors $e_i\cliff e_j\cliff\phi$ ($j\ne i$) are orthonormal, so $|\nabla^{LC}_{e_i}\phi|^2 = \frac14\sum_{j\neq i}\xi_j^2$; summing over $i$ gives $\frac74|\xi|^2$.
\end{proof}

\begin{rem}
The normalization of (\ref{eq:nablaphi_xi}) is independent of the sign conventions entering (\ref{eq:nablaLCphi}) and Theorem \ref{thm:ivanov}: by uniqueness of the characteristic connection, the right-hand sides of (\ref{eq:nablaLCphi}) and (\ref{eq:nablaphi_xi}) coincide, and only norms of these expressions enter the estimates below. Replacing $\theta$ by $-\theta$ (i.e.\ a different sign convention for the Lee form) changes nothing in what follows.
\end{rem}

\begin{proposition}
\label{prop:Atheta}
Let $v = a\cliff\phi$ and decompose $a = a_\parallel + a_\perp$ with $a_\parallel$ parallel and $a_\perp$ orthogonal to $\xi=\frac14\theta^\sharp$. Then
\begin{equation}
\label{eq:Aexplicit}
\mathcal{A}_\theta(v) = \frac{7}{2}\, g(a,\xi)\,\phi \;-\; \frac{5}{2}\,(a_\perp\wedge\xi)\cliff\phi ,
\end{equation}
and
\begin{equation}
\label{eq:Anorm}
|\mathcal{A}_\theta(v)|^2 = \frac{49}{4}\, g(a,\xi)^2 + \frac{25}{4}\,|a_\perp|^2|\xi|^2 \ \le\ \frac{49}{64}|\theta|^2\,|v|^2 ,
\end{equation}
with equality when $a$ is parallel to $\theta^\sharp$. Thus the pointwise squared operator norm of $\mathcal{A}_\theta$ is $c_\theta := \frac{49}{64}|\theta|^2$.
\end{proposition}

\begin{proof}
Insert (\ref{eq:nablaphi_xi}) in (\ref{eq:Adef}) and use the identities $\sum_i e_i\cliff a\cliff e_i = 6a$ (valid in dimension $8$) and $\sum_i g(e_i,\xi)\,e_i = \xi$:
$$\mathcal{A}_\theta(v) = -\frac12\Big(6\,a\cliff\xi + \xi\cliff a\Big)\cliff\phi = -\frac12\Big(5\,a\cliff\xi - 2g(a,\xi)\Big)\cliff\phi .$$
Since $a\cliff\xi = -g(a,\xi) + a_\perp\cliff\xi$ and $a_\perp\cliff\xi = a_\perp\wedge\xi$, we obtain (\ref{eq:Aexplicit}). The two terms are orthogonal because 2-forms act skew-adjointly, and $|(a_\perp\wedge\xi)\cliff\phi| = |a_\perp||\xi|$. Finally $|v| = |a|$ by Lemma \ref{lem:spinor_decomp}(c) and $g(a,\xi)^2 + |a_\perp|^2|\xi|^2 = |a|^2|\xi|^2$, which gives (\ref{eq:Anorm}).
\end{proof}

\begin{rem}
On an integrable $Spin(7)$ manifold ($\theta = 0$) one has $\mathcal{A}_\theta = 0$, and Corollary \ref{cor:H1} reduces to the well-known identification $\mathcal{H}^1\cong\Ker\dirac^-_A$. Note that by (\ref{eq:Aexplicit}) the component of $\mathcal{A}_\theta(v)$ along $\Lambda^2_7\cliff\phi$ is $-\frac52\,\pi_7(a_\perp\wedge\xi)\cliff\phi$.
\end{rem}

\section{Rigidity Analysis}

Let $\mathfrak{R}_A$ be the Clifford curvature endomorphism of $\spinS\otimes\gfrak_P$,
\begin{equation}
\label{eq:RA}
\mathfrak{R}_A(\sigma\otimes\eta) = \frac12\sum_{i,j=1}^8 e_i\cliff e_j\cliff\sigma\otimes[F_A(e_i,e_j),\eta],
\end{equation}
which is pointwise self-adjoint and preserves $\spinS^-\otimes\gfrak_P$. The Lichnerowicz--Weitzenb\"ock formula for the twisted Dirac operator reads \cite[Theorem II.8.17]{lawson_spin_2016}
\begin{equation}
\label{eq:weitzenbock}
\dirac_A^2 = (\conn^A)^*\conn^A + \frac14\Scal_g + \mathfrak{R}_A ;
\end{equation}
restricted to $\Gamma(\spinS^-\otimes\gfrak_P)$ the left-hand side is $\dirac^+_A\dirac^-_A$. Let $\lambda_A(x)$ denote the smallest eigenvalue of $\mathfrak{R}_A$ on $(\spinS^-\otimes\gfrak_P)_x$.

\begin{proposition}[A sufficient condition for rigidity]
\label{prop:rigidity}
Let $A$ be a $Spin(7)$-instanton on a compact LCP $Spin(7)$ manifold $M$. If
\begin{equation}
\label{eq:rigid_cond}
\lambda_A > \frac{49}{64}|\theta|^2 - \frac14\Scal_g = \frac{7}{64}|\theta|^2 - \frac78\,\delta\theta
\end{equation}
at every point of $M$, then $A$ is infinitesimally rigid, i.e.\ $\mathcal{H}^1 = 0$.
\end{proposition}

\begin{proof}
Let $v\in\Ker(\dirac^-_A-\mathcal{A}_\theta)$. Taking the $L^2$-inner product of (\ref{eq:weitzenbock}) with $v$ and using $(\dirac^-_A)^* = \dirac^+_A$,
$$\int_M |\mathcal{A}_\theta(v)|^2 = \int_M|\dirac^-_Av|^2 = \int_M|\conn^Av|^2 + \int_M\Big(\frac14\Scal_g|v|^2 + \langle\mathfrak{R}_Av,v\rangle\Big).$$
By Proposition \ref{prop:Atheta}, $|\mathcal{A}_\theta(v)|^2\le\frac{49}{64}|\theta|^2|v|^2$ pointwise, and $\langle\mathfrak{R}_Av,v\rangle\ge\lambda_A|v|^2$. Hence
$$0 \ \ge\ \int_M|\conn^Av|^2 + \int_M\Big(\lambda_A + \frac14\Scal_g - \frac{49}{64}|\theta|^2\Big)|v|^2 .$$
Under (\ref{eq:rigid_cond}) the integrand of the second term is non-negative and vanishes only where $v=0$; therefore $v\equiv0$. By Corollary \ref{cor:H1}, $\mathcal{H}^1=0$. The second expression in (\ref{eq:rigid_cond}) follows from (\ref{eq:ScalLCP}).
\end{proof}

\begin{rem}
\label{rem:sufficient}
Proposition \ref{prop:rigidity} is only a sufficient condition: if (\ref{eq:rigid_cond}) fails, no conclusion about $\mathcal{H}^1$ can be drawn from it. Non-rigidity has to be established by exhibiting non-zero elements of $\mathcal{H}^1$, as we do in Section \ref{sec:flat}. Commutativity of $G$ alone does not imply rigidity (see Theorem \ref{thm:nonrigid}).
\end{rem}

\begin{corollary}
\label{cor:flat_fails}
On a compact non-parallel LCP $Spin(7)$ manifold, condition (\ref{eq:rigid_cond}) is never satisfied by a flat connection.
\end{corollary}

\begin{proof}
If $F_A=0$ then $\mathfrak{R}_A=0$ and $\lambda_A=0$. Integrating, $\int_M\big(\frac{7}{64}|\theta|^2 - \frac78\delta\theta\big) = \frac7{64}\int_M|\theta|^2>0$ since $\int_M\delta\theta=0$ and $\theta\not\equiv0$. So the right-hand side of (\ref{eq:rigid_cond}) is positive somewhere.
\end{proof}

Thus the curvature of $A$ must compensate the torsion term $\frac{7}{64}|\theta|^2$ (in the sense of (\ref{eq:rigid_cond})) for the criterion to apply. Corollary \ref{cor:flat_fails} is consistent with, but does not imply, the following non-rigidity result.

\subsection{Flat connections are not rigid}
\label{sec:flat}

Let $P = M\times G$ be the trivial principal $G$-bundle and $A_0$ the product (trivial) connection, so that $\gfrak_P = M\times\gfrak$ is trivialised, $d_{A_0} = d$ acts componentwise, and $F_{A_0}=0$. In particular $A_0$ is a $Spin(7)$-instanton. It is reducible, with $\mathcal{H}^0 = \gfrak$ (the constant sections).

\begin{theorem}
\label{thm:nonrigid}
Let $M$ be a compact non-parallel LCP $Spin(7)$ manifold and $A_0$ the trivial connection on $P = M\times G$. Then
$$\mathcal{H}^1(A_0) \supseteq \mathcal{H}^1_{g}(M)\otimes\gfrak \cong H^1(M;\mathbb{R})\otimes\gfrak \neq 0,$$
where $\mathcal{H}^1_g(M)$ denotes the space of $g$-harmonic 1-forms. In particular $\dim\mathcal{H}^1(A_0)\ge b_1(M)\dim G\ge\dim G$, and $A_0$ is not rigid. Under the isomorphism of Corollary \ref{cor:H1}, these deformations correspond to the spinors $v = \sum_k (\alpha_k\cliff\phi)\otimes X_k$, with $\alpha_k$ harmonic and $X_k\in\gfrak$.

If moreover $G$ is abelian (e.g.\ $G = U(1)$), these infinitesimal deformations are unobstructed: the instanton moduli space near $[A_0]$ contains the torus of flat connections $H^1(M;\mathbb{R})\otimes\gfrak$ modulo periods, of dimension $b_1(M)\dim G$, whose tangent space at $[A_0]$ is $\mathcal{H}^1_g(M)\otimes\gfrak$.
\end{theorem}

\begin{proof}
Let $\alpha$ be a $g$-harmonic 1-form and $X\in\gfrak$, and put $a = \alpha\otimes X$. Then $d a = d\alpha\otimes X = 0$, so $\pi_7(da)=0$, and $d^*a = d^*\alpha\otimes X = 0$. Hence $a$ lies in the right-hand side of (\ref{eq:hodge}), and is non-zero if $\alpha\neq0$. By the Hodge theorem $\mathcal{H}^1_g(M)\cong H^1(M;\mathbb{R})$, and $b_1(M)\geq1$ since $[\theta]\neq0$ (equivalently, by \cite[Theorem B]{ivanov_parton_piccinni_2006}, $M$ fibres over $S^1$). This proves the first part.

If $G$ is abelian, then for every closed $\gfrak$-valued 1-form $a$ the connection $A_0 + a$ is flat, hence a $Spin(7)$-instanton. Two such connections $A_0+a$, $A_0+a'$ are gauge equivalent if and only if $a - a'$ is exact up to a period in the integral lattice of $H^1(M;\mathbb{R})\otimes\gfrak$ determined by $G$. Hence $\{[A_0 + a] : a\in\mathcal{H}^1_g(M)\otimes\gfrak\}$ is a torus of dimension $b_1(M)\dim G$ of instantons through $[A_0]$, with tangent space $\mathcal{H}^1_g(M)\otimes\gfrak$.
\end{proof}

\begin{rem}
\label{rem:flatrems}
\begin{enumerate}
\item[(a)] For non-abelian $G$, the curves $s\mapsto A_0 + s\,\alpha\otimes X$ with $\alpha$ closed are still curves of flat connections (since $[X,X]=0$), so the first-order deformations $\alpha\otimes X$ again integrate. However, $A_0$ is reducible ($\mathcal{H}^0=\gfrak$), and the local structure of $\mathcal{M}$ near $[A_0]$ is that of a Kuranishi model divided by the stabilizer $G$; we do not claim that $\mathcal{M}$ is smooth there.
\item[(b)] On an integrable compact $Spin(7)$ manifold, $\pi_7(da)=0$ forces $da=0$ (since $\int_M da\wedge da\wedge\Phi = 0$), so that $\mathcal{H}^1(A_0) = \mathcal{H}^1_g(M)\otimes\gfrak$. In the LCP case the corresponding integral equals $\int_M a\wedge da\wedge\theta\wedge\Phi$, which need not vanish. Nevertheless, for the product manifolds $N\times S^1$ the inclusion of Theorem \ref{thm:nonrigid} is an equality (Corollary \ref{cor:trivial_product}). Whether equality holds for all compact LCP $Spin(7)$ manifolds remains open.
\end{enumerate}
\end{rem}

\begin{ex}[The Hopf manifold $S^7\times S^1$]
Consider Example \ref{ex:hopf} with $P = M\times U(1)$ and the trivial connection $A_0$. By Theorem \ref{thm:nonrigid}, $\mathcal{H}^1(A_0)$ contains the harmonic form $dt = -\frac14\theta$, and the circle of flat connections $A_0 + s\,i\,dt$, $s\in\mathbb{R}/(2\pi/\log\lambda)\mathbb{Z}$, consists of pairwise gauge-inequivalent flat instantons (their holonomies around the $S^1$-factor are distinct).

This example also illustrates Corollary \ref{cor:flat_fails} and the sharpness of the constants. Here $|\xi| = 1$, $\Scal_g = 42$, and $\frac{49}{64}|\theta|^2 - \frac14\Scal_g = \frac{49}{4} - \frac{21}{2} = \frac74 > 0$, so (\ref{eq:rigid_cond}) fails, as it must. For $a = dt$ (which is parallel for the product metric and parallel to $\xi$) and $v = a\cliff\phi$, Lemma \ref{lem:nablaphi} and Proposition \ref{prop:Atheta} give
$$|\conn v|^2 = |\nabla^{LC}\phi|^2 = \frac74,\qquad |\mathcal{A}_\theta(v)|^2 = \frac{49}{4},$$
and the integrated Weitzenb\"ock identity $\int|\mathcal{A}_\theta(v)|^2 = \int|\conn v|^2 + \frac14\int\Scal_g|v|^2$ holds pointwise in the form $\frac{49}{4} = \frac74 + \frac{21}{2}$.
\end{ex}

\section{Product LCP manifolds and pull-back instantons}
\label{sec:product}

In this section we treat the basic class of compact LCP $Spin(7)$ manifolds and show that, for instantons pulled back from the fibre, the deformation theory reduces completely to the nearly parallel $G_2$ case.

\subsection{The setting}
Let $(N,\varphi)$ be a compact nearly parallel $G_2$ manifold, normalised so that
$$d\varphi = 4\psi,\qquad \psi = *_N\varphi .$$
In particular $d\psi=0$, $g_N$ is Einstein with $\mathrm{Ric}_N = 6g_N$, and $b_1(N)=0$ by Bochner's theorem. On $M = N\times S^1_\ell$, $S^1_\ell=\mathbb{R}/\ell\mathbb{Z}$, consider the product metric $g = dt^2+g_N$ and the 4-form
$$\Phi = dt\wedge\varphi + \psi .$$
Pointwise $(\Phi,g)$ is modelled on $(\Phi_0,g_0)$ of Section 2 (indeed $\Phi_0 = dx^1\wedge\varphi_0 + *_7\varphi_0$ with $\varphi_0 = \partial_{x^1}\hook\Phi_0$), so $\Phi$ is a $Spin(7)$-structure inducing $g$. Since $d\Phi = -dt\wedge d\varphi = -4\,dt\wedge\Phi$, it is LCP with Lee form
$$\theta = -4\,dt,$$
which is parallel and not exact; thus $M$ is a compact non-parallel LCP $Spin(7)$ manifold with $b_1(M)=1$, $|\theta|^2=16$, $\delta\theta=0$ and $\Scal_g = 42$, in accordance with (\ref{eq:ScalLCP}). Locally $g$ is conformal to the cone metric $dr^2+r^2g_N$, $r=e^t$, whose holonomy is contained in $Spin(7)$. For $N=S^7$ with its standard nearly parallel structure this is exactly the Hopf manifold of Example \ref{ex:hopf}.

\begin{rem}
\label{rem:conformal}
The deformation complex (\ref{eq:complex}) depends on the $Spin(7)$-structure only through the subbundle $\Lambda^2_7$, which is unchanged under conformal rescalings $\Phi\mapsto e^{4u}\Phi$. Hence $\mathcal{H}^0$, $\mathcal{H}^1$ and $\mathcal{H}^2$ are invariants of the conformal class of the LCP structure, and in computing them we are free to use any metric in this class (the identification (\ref{eq:hodge}) and Proposition \ref{prop:rigidity} then refer to that metric). On $N\times S^1$ we use the product metric, for which $\theta$ is parallel.
\end{rem}

We write $\Lambda^k_N$ for forms on $M$ with no $dt$-component. Recall that for the $G_2$-structure $\varphi$ one has $\Lambda^2_N = \Lambda^2_7(N)\oplus\Lambda^2_{14}(N)$, and that a 2-form $\gamma\in\Lambda^2_N$ lies in $\Lambda^2_{14}(N)$ if and only if $\gamma\wedge\psi = 0$.

\begin{lemma}
\label{lem:cylinder}
Let $\beta\in\Lambda^1_N$ and $\gamma\in\Lambda^2_N$. Then
$$\pi_7(dt\wedge\beta + \gamma) = 0 \iff \beta = -*_N(\gamma\wedge\psi).$$
\end{lemma}

\begin{proof}
For $\alpha\in\Lambda^1_N$ put $\omega_\alpha = dt\wedge\alpha + \alpha\hook\varphi$. One checks from $\Lambda^2_7=\{\omega : *(\omega\wedge\Phi) = 3\omega\}$ that $\Lambda^2_7 = \{\omega_\alpha : \alpha\in\Lambda^1_N\}$. Hence $\pi_7(dt\wedge\beta+\gamma)=0$ if and only if, for all $\alpha$,
$$0 = \langle dt\wedge\beta + \gamma, \omega_\alpha\rangle = \langle\beta,\alpha\rangle + \langle\gamma,\alpha\hook\varphi\rangle = \langle\beta + *_N(\gamma\wedge\psi),\alpha\rangle,$$
where we used $\alpha\hook\varphi = *_N(\alpha\wedge\psi)$ and $\langle\gamma,*_N(\alpha\wedge\psi)\rangle\,\mathrm{vol}_N = \alpha\wedge\gamma\wedge\psi = \langle\alpha,*_N(\gamma\wedge\psi)\rangle\,\mathrm{vol}_N$. (We have also verified the lemma, including the sign, by a direct computation with the explicit form $\Phi_0$ of Section 2.)
\end{proof}

Let $Q\to N$ be a principal $G$-bundle with a connection $B$. Recall that $B$ is a \emph{$G_2$-instanton} if $\pi_7(F_B)=0$, i.e.\ $F_B\wedge\psi=0$. Since $d\psi=0$, the $G_2$-instanton deformation complex
$$0\to\Omega^0(N,\gfrak_Q)\xrightarrow{d_B}\Omega^1(N,\gfrak_Q)\xrightarrow{\pi_7\circ d_B}\Omega^2_7(N,\gfrak_Q)\to0$$
is an elliptic complex \cite{reyes_carrion_1998, singhal_deformations_2022}; we denote its cohomology by $\mathcal{H}^i_{G_2}(B)$. As in Section 4, Hodge theory gives
$$\mathcal{H}^0_{G_2}(B)=\{f : d_Bf=0\},\qquad \mathcal{H}^1_{G_2}(B)\cong\{b\in\Omega^1(N,\gfrak_Q) : d_B^*b = 0,\ d_Bb\wedge\psi = 0\}.$$

\subsection{The reduction theorem}

\begin{theorem}
\label{thm:product}
Let $B$ be a $G_2$-instanton on $Q\to N$, and let $A=\mathrm{pr}^*B$ be its pull-back to $P=\mathrm{pr}^*Q\to M=N\times S^1_\ell$. Then $A$ is a $Spin(7)$-instanton, and the map $(f,b)\mapsto f\,dt + b$ induces an isomorphism
$$\mathcal{H}^0_{G_2}(B)\oplus\mathcal{H}^1_{G_2}(B)\ \xrightarrow{\ \cong\ }\ \mathcal{H}^1_{Spin(7)}(M,A).$$
\end{theorem}

\begin{proof}
Since $F_A = \mathrm{pr}^*F_B$ has no $dt$-component, Lemma \ref{lem:cylinder} with $\beta=0$ shows that $\pi_7(F_A)=0$ if and only if $F_B\wedge\psi=0$. So $A$ is a $Spin(7)$-instanton.

By (\ref{eq:hodge}) and Remark \ref{rem:conformal}, $\mathcal{H}^1(M,A)$ is the space of solutions $a$ of (\ref{eq:lin_inst})--(\ref{eq:gauge_fix}) for the product metric. Write $a = f\,dt + b$, where $f(t)\in\Omega^0(N,\gfrak_Q)$ and $b(t)\in\Omega^1(N,\gfrak_Q)$ depend smoothly and $\ell$-periodically on $t$. Since $A$ has no $dt$-component,
$$d_Aa = dt\wedge(\partial_tb - d_Bf) + d_Bb,\qquad d_A^*a = -\partial_tf + d_B^*b .$$
By Lemma \ref{lem:cylinder}, the system (\ref{eq:lin_inst})--(\ref{eq:gauge_fix}) is therefore equivalent to the evolution equation
\begin{equation}
\label{eq:evolution}
\partial_t u = L_Bu,\qquad u = (f,b),\qquad L_B(f,b) = \big(d_B^*b,\ d_Bf - *_N(\psi\wedge d_Bb)\big).
\end{equation}
The operator $L_B$ on $\Omega^0(N,\gfrak_Q)\oplus\Omega^1(N,\gfrak_Q)$ is formally self-adjoint: the off-diagonal entries $d_B$, $d_B^*$ are adjoint to each other, and for $b,c\in\Omega^1(N,\gfrak_Q)$, using $d\psi=0$ and the $\Ad$-invariance of the inner product,
$$\int_N\langle *_N(\psi\wedge d_Bb),c\rangle = \int_N\langle c\wedge\psi\wedge d_Bb\rangle = \int_N\langle d_Bc\wedge\psi\wedge b\rangle = \int_N\langle b, *_N(\psi\wedge d_Bc)\rangle,$$
where the middle equality follows by integrating $d\langle c\wedge\psi\wedge b\rangle$ over $N$.

Expand $u$ in a Fourier series $u = \sum_{k\in\mathbb{Z}} e^{2\pi i kt/\ell}u_k$ with $u_k$ smooth sections of the complexified bundle; the series converges in $C^\infty$. Equation (\ref{eq:evolution}) becomes $\frac{2\pi i k}{\ell}u_k = L_Bu_k$ for every $k$. Taking the $L^2$-inner product with $u_k$, the right-hand side $\langle L_Bu_k,u_k\rangle_{L^2}$ is real, while the left-hand side is purely imaginary. Hence $u_k=0$ for $k\neq0$, so $u = u_0$ is independent of $t$ and $L_Bu_0=0$.

It remains to compute $\Ker L_B$. Let $L_B(f,b)=0$, i.e.\ $d_B^*b=0$ and $d_Bf = *_N(\psi\wedge d_Bb)$. In dimension 7, $d_B^* = -*_Nd_B*_N$ on 1-forms and $*_N^2=1$, so
$$d_B^*d_Bf = -*_Nd_B(\psi\wedge d_Bb) = -*_N\big(\psi\wedge[F_B\wedge b]\big) = 0,$$
because $d\psi=0$ and $F_B\wedge\psi=0$. Integrating against $f$ gives $d_Bf=0$, and then $*_N(\psi\wedge d_Bb)=0$. Thus $f\in\mathcal{H}^0_{G_2}(B)$ and $b$ lies in the harmonic model of $\mathcal{H}^1_{G_2}(B)$. Conversely, every such pair $(f,b)$ lies in $\Ker L_B$, and $f\,dt+b$ solves (\ref{eq:lin_inst})--(\ref{eq:gauge_fix}).
\end{proof}

\begin{rem}
Formally, (\ref{eq:evolution}) is the temporal-gauge form of the $Spin(7)$ deformation operator on the cylinder, $\partial_t - L_B$, familiar from the analysis of instantons on manifolds with cylindrical ends. Here periodicity in $t$ is enough to kill all non-zero modes; no spectral gap for $L_B$ is needed. The proof uses neither the Dirac reformulation nor any curvature assumption, only $d\psi = 0$ and the product structure. The cylindrical form $\partial_t + L_{B,t}$ of the $Spin(7)$ deformation operator for a general family of cocalibrated $G_2$-structures $\Phi = dt\wedge\varphi_t+\psi_t$, together with the fact that $L_{B,t}$ is formally self-adjoint precisely when $d\psi_t = 0$, is established in the companion paper \cite{yalcinkaya_heisenberg}, where it is used for an Atiyah--Patodi--Singer index computation on half-cylinders with non-vanishing $W_{48}$-torsion. Theorem \ref{thm:product} is the compact, periodic and locally conformally parallel counterpart, in which the operator is $t$-independent and the full deformation space can be determined.
\end{rem}

\begin{corollary}
\label{cor:trivial_product}
For the trivial connection $A_0$ on $M\times G = \mathrm{pr}^*(N\times G)$,
$$\mathcal{H}^1(M,A_0)\cong\gfrak = H^1(M;\mathbb{R})\otimes\gfrak,$$
spanned by the deformations $X\,dt$, $X\in\gfrak$. In particular the inclusion of Theorem \ref{thm:nonrigid} is an equality on $N\times S^1$.
\end{corollary}

\begin{proof}
For the trivial connection, $\mathcal{H}^0_{G_2}=\gfrak$, and it suffices to show that $\mathcal{H}^1_{G_2}=0$, i.e.\ that a real 1-form $b$ with $d^*b=0$ and $db\wedge\psi=0$ vanishes (cf.\ \cite[abelian rigidity]{singhal_deformations_2022}). Since $db\in\Lambda^2_{14}(N)$, we have $*_N(db\wedge\varphi) = -db$, hence
$$\|db\|^2_{L^2} = -\int_N db\wedge db\wedge\varphi = -\int_N b\wedge db\wedge d\varphi = -4\int_N b\wedge (db\wedge\psi) = 0,$$
where the second equality follows by integrating $d(b\wedge db\wedge\varphi)$. Thus $db=0=d^*b$, so $b$ is harmonic, and $b=0$ since $b_1(N)=0$.
\end{proof}

\begin{corollary}
\label{cor:product_rigid}
$\mathrm{pr}^*B$ is infinitesimally rigid if and only if $B$ is irreducible ($\mathcal{H}^0_{G_2}(B)=0$) and infinitesimally rigid as a $G_2$-instanton ($\mathcal{H}^1_{G_2}(B)=0$). If $B$ is reducible, then every $f\in\mathcal{H}^0_{G_2}(B)$ gives an honest curve of $Spin(7)$-instantons $s\mapsto A + s f\,dt$ on $M$, since $F_{A+sf\,dt} = F_A - s\,dt\wedge d_Bf = F_A$.
\end{corollary}

\begin{proof}
The first statement is immediate from Theorem \ref{thm:product}. For the second, $d_Bf=0$ and $[f\,dt\wedge f\,dt]=0$ give $F_{A+sf\,dt}=F_A$, which satisfies $\pi_7(F_A)=0$.
\end{proof}

\begin{ex}[Non-flat instantons on $S^7\times S^1$]
\label{ex:canonical}
Let $N = S^7 = Spin(7)/G_2$ with its standard nearly parallel $G_2$-structure, so that $M = S^7\times S^1$ is the Hopf manifold of Example \ref{ex:hopf}. Let $B$ be the canonical connection on the principal $G_2$-bundle $Spin(7)\to S^7$. Its curvature is $F_B(X,Y) = -[X,Y]_{\mathfrak{g}_2}$ for $X,Y\in\mathfrak{m}$, where $\mathfrak{spin}(7)=\mathfrak{g}_2\oplus\mathfrak{m}$ and $\mathfrak{m}\cong\mathbb{R}^7$; the map $\Lambda^2\mathfrak{m}=\Lambda^2_7\oplus\Lambda^2_{14}\to\mathfrak{g}_2$ is $G_2$-equivariant, so by Schur's lemma it vanishes on $\Lambda^2_7$, i.e.\ $B$ is a $G_2$-instanton (cf.\ \cite{harland_instantons_2012, singhal_deformations_2022}). Its holonomy algebra contains $[\mathfrak{m},\mathfrak{m}]_{\mathfrak{g}_2}$; since $\mathfrak{m}+[\mathfrak{m},\mathfrak{m}]$ is an ideal of the simple Lie algebra $\mathfrak{spin}(7)$, it equals $\mathfrak{spin}(7)$, so $[\mathfrak{m},\mathfrak{m}]_{\mathfrak{g}_2}=\mathfrak{g}_2$ and the holonomy algebra is all of $\mathfrak{g}_2$. Hence $B$ is irreducible and not flat. By Theorem \ref{thm:product}, $A=\mathrm{pr}^*B$ is a non-flat LCP $Spin(7)$-instanton on $S^7\times S^1$ with structure group $G_2$, and
$$\mathcal{H}^1_{Spin(7)}(S^7\times S^1,A)\cong\mathcal{H}^1_{G_2}(S^7,B).$$
The same argument applies to the canonical connections on the other normal homogeneous nearly parallel $G_2$ manifolds. The spaces $\mathcal{H}^1_{G_2}(B)$ for the canonical connections on the four normal homogeneous nearly parallel $G_2$ manifolds were described by Singhal \cite{singhal_deformations_2022}. Via Theorem \ref{thm:product}, her results therefore determine the infinitesimal deformations of the corresponding LCP $Spin(7)$-instantons on $N\times S^1$, and in particular decide their rigidity.
\end{ex}

\begin{rem}[Mapping tori]
By \cite{ivanov_parton_piccinni_2006}, general compact LCP $Spin(7)$ manifolds fibre over $S^1$ with nearly parallel $G_2$ fibres. When $M$ is the mapping torus $(N\times\mathbb{R})/\mathbb{Z}$ of an automorphism $\phi$ of the nearly parallel $G_2$-structure, with $\mathbb{Z}$ acting by $(x,t)\mapsto(\phi(x),t+\ell)$, and $\phi$ lifts to an automorphism $\hat\phi$ of $(Q,B)$, the proof of Theorem \ref{thm:product} applies with the periodicity condition replaced by $u(t+\ell) = \hat\phi^*u(t)$. Since $L_B$ is self-adjoint and elliptic with discrete spectrum, and $\hat\phi^*$ is unitary and commutes with $L_B$, we may expand $u$ in eigenspaces of $L_B$: the $\lambda$-component satisfies $u_\lambda(t)=e^{\lambda t}u_\lambda(0)$ and $\|u_\lambda(t+\ell)\|=\|\hat\phi^*u_\lambda(t)\|=\|u_\lambda(t)\|$, which forces $\lambda=0$ or $u_\lambda=0$. Thus $u$ is constant in $t$, lies in $\Ker L_B$ and is $\hat\phi$-invariant. Hence $\mathcal{H}^1(M,A)\cong\big(\mathcal{H}^0_{G_2}(B)\oplus\mathcal{H}^1_{G_2}(B)\big)^{\hat\phi}$. In particular, $\mathcal{H}^1(A_0)=\gfrak=H^1(M;\mathbb{R})\otimes\gfrak$ for the trivial connection on every such mapping torus.
\end{rem}

\subsection{Further directions}
Several natural questions remain. (i) The most interesting instantons on compact LCP $Spin(7)$ manifolds are those that are not pulled back from the fibre. For these, Theorem \ref{thm:product} does not apply, and $\Ker(\dirac^-_A-\mathcal{A}_\theta)$ has to be studied directly; Proposition \ref{prop:rigidity} gives a first vanishing criterion. (ii) The obstruction space $\mathcal{H}^2$ should be analysed in the same spirit. In temporal form it is governed by the cokernel of $\partial_t-L_B$, which suggests that pulled-back instantons are unobstructed whenever $B$ is unobstructed as a $G_2$-instanton. (iii) It would be interesting to relate the moduli spaces on $N\times S^1$ to the moduli of $Spin(7)$-instantons on the cone $C(N)$, which is conformally equivalent to the universal cover $N\times\mathbb{R}$. Homogeneous examples beyond products are severely restricted: compact non-parallel LCP $Spin(7)$ manifolds are never simply connected, and, for instance, the full flag manifold $Sp(2)/T^2$ does not even admit a $Spin(7)$-structure \cite[Theorem A]{alekseevsky_homogeneous_2020}.

\end{document}